\RequirePackage[l2tabu, orthodox]{nag}
\documentclass[12pt,letterpaper]{article}
\usepackage{amsfonts}
\usepackage{amsmath}
\usepackage{amssymb}
\usepackage{amsthm}
\usepackage{enumerate}
\usepackage{color}
\usepackage{geometry}
\usepackage{hyperref}
\allowdisplaybreaks
\newtheorem{theorem}{Theorem}[section]

\newtheorem{lemma}[theorem]{Lemma}

\newcommand{\Supp}{\operatorname{Supp} }

\newcommand{\lt}{\overline}

\newcommand{\lc}{\operatorname{lc}}
\newcommand{\Irr}{\operatorname{Irr}}
\newcommand{\IFF}{if and only if}

\newcommand{\NN}{\mathbb{N}}

\newcommand{\ZZ}{\mathbb{Z}}

\newcommand{\gsb}{Gr\"obner-Shirshov basis}
\newcommand{\gsbs}{Gr\"obner-Shirshov bases}

\title{Jacobson's lemma via Gr\"obner-Shirshov bases}
\author{Xiangui Zhao\\
\small Department of Mathematics,
Huizhou University\\
\small Huizhou, Guangdong Province 516007, China\\
\small zhaoxg@hzu.edu.cn}
\date{}                                           
\begin{document}
\maketitle
\textbf{Abstract.} Let $R$ be a ring with identity $1$.
{Jacobson's lemma} states that for any $a,b\in R$, if $1-ab$ is invertible then so is $1-ba$.
Jacobson's lemma has suitable analogues for several types of generalized inverses, e.g., Drazin inverse,
generalized Drazin inverse, and inner inverse.
In this note we give a constructive way via Gr\"obner-Shirshov basis theory to obtain the inverse of $1-ab$ in terms of $(1-ba)^{-1}$, assuming the latter exists.

\textbf{2010 Mathematics Subject Classification:}
15A09, 
13P10,

\textbf{Keyword:} inverse, \gsb, Jacobson's lemma
\section{Introduction}
Let $R$ be a ring with identity $1$.
Recall that an element $x\in R$ is \emph{invertible (in $R$)} if there exists $y\in R$ such that $xy=yx=1$.
{Jacobson's lemma} \cite{lam2014jacobson,corach2013extensions} states that for any $a,b\in R$, if $1-ab$ is invertible then so is $1-ba$.
Jacobson's lemma has suitable analogues for several types of generalized inverses, e.g., Drazin inverse \cite{castro2010generalized, cvetkovic2010jacobson},
generalized Drazin inverse \cite{zhuang2012jacobson}, and inner inverse \cite{chen2010unit}.

Suppose $a, b\in R$ and $1-ab$ is invertible.
It is easy to verify that
\begin{align*}
(1 + b (1 - ab)^{-1}a)(1-ba)=1-b(1-(1-ab)^{-1}+(1-ab)^{-1}ab)a=1-b0a=1\\
(1-ba)(1 + b (1 - ab)^{-1}a)=1+b((1-ab)^{-1}-1-ab(1-ab)^{-1}a)a=1+b0a=1
\end{align*}
 and then Jacobson's lemma follows.
But, how is this formula constructed?
As Halmos \cite{halmos1981does} observed, $(1-ba)^{-1}$ can be ``constructed'' by using the following ``geometric series trick'',
which is ``usually ascribed to Jacobson'' \cite{lam2014jacobson},
\begin{align}\label{formula_Jacabson}
  (1-ba)^{-1}&=1 + ba + baba + bababa + \cdots\nonumber \\
  &=1 + b (1 + ab + abab + \cdots) a\nonumber \\
  &= 1 + b (1 - ab)^{-1}a.
\end{align}
Once the inverse is ``constructed'', it is easy to verify that Formula (\ref{formula_Jacabson}) works as above, despite its unlawful derivation.

Now it is natural to ask: Does there exist a legal way to construct the inverse of $1-ba$?
We shall answer this question in this note.

In this note we give a constructive way via \gsb\ theory to obtain the inverse of $1-ab$ in terms of $(1-ba)^{-1}$, assuming the latter exists.
A similar method is used in \cite{herfort2014grobner} to obtain Hua's identity.

\section{\gsbs\ for algebras over a unitary ring}
In this section, we briefly review several properties of \gsb\ for algebras over a unitary commutative ring.

Let $X\neq\emptyset$ and $k$ be a unitary ring.
Let $X^*$ be the free monoid generated by $X$ and $A=k\langle X\rangle =kX^*$ be the free algebra over $k$ with generator set $X$.
Similar to the case of associative algebras over a field (see \cite{bokut2014grobner}, c.f., \cite{cox2005using,zhangzhao2013Modules}),
we have concepts of \emph{monomial ordering},
\emph{leading term/coefficient}, \emph{composition}, and \emph{\gsb}.
Particularly, a \gsb\ consists of only \emph{monic polynomials} (i.e., polynomials with leading coefficient $1$).
For convenience, we define the leading term of a nonzero element from $k$ to be $1$ and the leading term of $0$ to be $0$.
Fixing a monomial ordering, we denote the leading term and leading coefficient of $f\in A$ by $\lt{f}$ and $\lc(f)$ respectively.

Given $f=\alpha_1u_1+\cdots\alpha_nu_n\in A$, where $n\in \NN$, each $\alpha_i\in k$, $u_i\in X^*$, and $\alpha_i\neq0$,
the \emph{support} of $f$ is $\Supp(f)=\{u_1,\ldots,u_n\}$.
Let $f,h, g\in A$, where $g$ is a monic polynomial.
Then \emph{$f$ reduces to $h$ modulo $g$},
denoted by $f\rightarrow_g h$, if $\lt{f}=\lt{agb}$ for some $a,b\in X^*$,
and $h=f-\lc(f) agb$.
For a set $G$ consisting of monic polynomials,
we say that \emph{$f$ reduces to $h$ modulo $G$},
denoted by  $f\rightarrow_G {h}$, if there exists a finite chain of reductions
\[
f\rightarrow_{g_1} f_1\rightarrow_{g_2} f_2\rightarrow_{g_3}\cdots \rightarrow_{g_t} f_t=h,
\]
where each $g_i\in G$ and $t\in \NN$.
Denote the set of words that are {\it irreducible} w.r.t. $G$ by $\Irr(G)$, i.e.,
\[
\Irr(G)=\{w\in X^*:w\not=a\lt{g}b \mbox{ for any }g\in G, a,b\in X^*\}.
\]
If {${\Supp(h)}\subseteq \Irr(G)$}, then we say $h$ is  \emph{irreducible} w.r.t. $G$.

The following lemma follows from the composition-diamond lemma for associative algebras over a commutative ring (\cite{bokut2014grobner}, Theorem 1 and Remark 2).
\begin{lemma}\label{lemma-GSB-reduction}
  Suppose $G$ is a \gsb\ for $I\unlhd A$, where $I$ is the ideal of $A$ generated by $G$.
  Then, for any $f\in A$, $f\in I$ \IFF\ $f\to_G0$.
\end{lemma}
\section{A constructive method to Jacobson's lemma}


Let $R$ be a unitary ring.
Given $a,b\in R$ with $1-ab$ invertible (say, $c=(1-ab)^{-1}$),
we give a constructive method in this section to find the inverse of $1-ba$.

Let $S$ be the $\ZZ$-algebra generated by $\{a,b,c\}$ with defining relations
$\{(1-ab)c=1, c(1-ab)=1\}$, i.e.,
\(
S=
\ZZ\langle a,b,c: (1-ab)c=1, c(1-ab)=1\rangle.
\)
We 
want to solve the following system for $x$ in $S$:
\begin{align}\label{system}
 (1-ba)x=1\nonumber\\
  x(1-ba)=1.
\end{align}
It is obvious that the subring of $R$ generated by $\{a,b,c\}$ is a ring homomorphic image of $S$. Thus the image of $x$ is the inverse of $1-ba$ in $R$.

Let $<$ be the degree-lexicographic ordering on $\{a,b,c\}^*$ with $a<b<c$.
Denote $f=abc-c+1$, $g=cab-c+1$, and $G=\{abc-c+1, cab-c+1\}=\{f,g\}$.

\

The following two lemmas will be used in the proof of our main theorem.
\begin{lemma}
  The set $G$ is a \gsb\ (in $\ZZ\langle a,b,c\rangle$) for $S$ w.r.t. $<$.
\end{lemma}
\begin{proof}
  There exist only  two compositions (both are intersection compositions) in $G$, i.e.,
  \begin{align*}
    fab-abg=(abc-c+1)ab-ab(cab-c+1)=-cab+abc\rightarrow_G 0
  \end{align*}
and
\begin{align*}
    gc-cf=(cab-c+1)c-c(abc-c+1)=0.
  \end{align*}
Hence $G$ is a \gsb.
\end{proof}
\begin{lemma}\label{lemma_leading-reducible}
  Suppose $\sum_{i=1}^n\alpha_i uv_i\to_G0$, where
  $1\leq n\in\NN$,  $0\neq\alpha_i\in \ZZ$, $u, v_i\in\Irr(G)$ for all $1\leq i\leq n$ and
  $v_1>v_2>\cdots>v_n$.
  Then $uv_1\not\in \Irr(G)$.
\end{lemma}
\begin{proof}
  If $uv_1\in \Irr(G)$, then $uv_1\neq 0$ is the leading term of $\sum_{i=1}^n\alpha_iuv_i$,
  contradicting the assumption that $\sum_{i=1}^n\alpha_iuv_i\to_G0$.
\end{proof}

The following theorem gives a new proof for Jacobson's lemma.
\begin{theorem}\label{thm_main}
  System (\ref{system}) has a unique solution in $S$.
\end{theorem}
\begin{proof}
(Uniqueness) It follows from the uniqueness of an inverse in a ring.

(Existence)
Suppose a solution $x\in S$ of the system exists and let $x=\sum_{i=1}^m\alpha_iu_i$ where $0\neq\alpha_i\in\ZZ$, $m\in \NN$ and all $u_i\in{\Irr(G)}$ with $u_1>u_2>\cdots>u_m$. Then it follows from $(1-ba)x=1$ and Lemma \ref{lemma-GSB-reduction} that
\begin{align*}
  bax-x+1=\sum_{i=1}^m \alpha_ibau_i-\sum_{i=1}^m\alpha_iu_i+1\to_G0.
\end{align*}
If $bau_1\in\Irr(G)$ then $bau_1\neq0$ is the leading term of $bax-x+1$,
contradicting the fact that $bax-x+1\to_G0$.
Thus $bau_1\not\in\Irr(G)$.
Hence $u_1=bcv$ for some $v\in \{a,b,c\}^*$,
where $\{a,b,c\}^*$ denotes the free monoid generated by $\{a,b,c\}$.
Similarly, the assumption $1-x(1-ba)=0$ in $S$ implies that $u_1=wca$ for some $w\in \{a,b,c\}^*$.

Now there are two cases for $u_1$: (i) $u_1=bca$ and (ii) $u_1=bczca$ for some $z\in \{a,b,c\}^*$.
We begin with Case (i).
Suppose $u_1=bca$.
First we notice that
\[
ba(\alpha_1bca)-\alpha_1bca+1\to_f \alpha_1b(c-1)a-\alpha_1bca+1=-\alpha_1ba+1,
\]
where $-\alpha_1ba+1$ is irreducible modulo $G$.
Thus $ba(\alpha_1bca)-\alpha_1bca+1\not\to_G 0$ and $x\neq \alpha_1 bca$.
Hence $m\geq 2$.
Now we may suppose $x=\alpha_1bca+\sum_{i=2}^m\alpha_iu_i$.
Then
\begin{align}
  bax-x+1&= \alpha_1babca+\sum_{i=2}^m\alpha_ibau_i-\alpha_1bca-\sum_{i=2}^m\alpha_iu_i+1\nonumber\\
  &\rightarrow_f \alpha_1b(c-1)a+\sum_{i=2}^m\alpha_ibau_i-\alpha_1bca-\sum_{i=2}^m\alpha_iu_i+1\nonumber\\
  &= \sum_{i=2}^m\alpha_ibau_i -\alpha_1ba -\sum_{i=2}^m\alpha_iu_i+1\label{eqn_m}.
\end{align}
We claim that $m=2$.
Otherwise, suppose $m\geq 3$.
Then $u_2>u_3\geq1$, $bau_2>bau_i$, $bau_2>ba$,
and $bau_2>u_j$ for all $i\geq 3$ and $j\geq 2$.
Namely, $bau_2$ is the leading monomial of (\ref{eqn_m}).
Thus the fact $bax-x+1\to_G0$ implies $bau_2\not\in\Irr(G)$.
So $u_2=bcu'$ for some $u'\in \{a,b,c\}^*$.
But $u_2<u_1=bca$.
Hence $u'=1$ and $u_2=bc$. 
If $bau_i\not\in\Irr(G)$ for some $i\geq 3$, then by the same argument we have $u_i=bc$,
which contradicts that $u_2>u_i$.
Therefore $bau_i\in\Irr(G)$ for $i\geq 3$.
Thus we may reduce polynomial (\ref{eqn_m}) as follows
\begin{align*}
&\sum_{i=2}^m\alpha_ibau_i -\alpha_1ba -\sum_{i=2}^m\alpha_iu_i+1\\
=&\ \alpha_2babc+\sum_{i=3}^m\alpha_ibau_i -\alpha_1ba -\alpha_2bc-\sum_{i=3}^m\alpha_iu_i+1\\
\to_f&\ \alpha_2b(c-1)+\sum_{i=3}^m\alpha_ibau_i -\alpha_1ba -\alpha_2bc-\sum_{i=3}^m\alpha_iu_i+1\\
=&\  -\alpha_2b+\sum_{i=3}^m\alpha_ibau_i -\alpha_1ba -\sum_{i=3}^m\alpha_iu_i+1.
\end{align*}
The last polynomial, denoted by $h$, is irreducible and its leading term is
(note that if $u_3=1$ then $m=3$)
\[
\lt{h}=\begin{cases}
  bau_3 &\text{ if } u_3>1\\
  ba &\text{ if } u_3=1 \text{ and } \alpha_3\neq \alpha_1\\
  b &\text{ if } u_3=1 \text{ and }  \alpha_3= \alpha_1
\end{cases}.
\]
In each of these three cases, $\lt{h}\neq0$ and thus $h\neq 0$,
which contradicts the assumption that $bax-x+1\rightarrow_G 0$.
This proves our claim that $m=2$.

Now we may write $x=\alpha_1bca+\alpha_2u_2$ and from (\ref{eqn_m}) we have
\begin{align*}
  bax-x+1\to_f\alpha_2bau_2-\alpha_1ba-\alpha_2u_2+1.
\end{align*}
If $bau_2\not\in \Irr(G)$, then $u_2=bc$ as we proved in the last paragraph.
Then we have
\begin{align*}
  bax-x+1&\to_f\alpha_2babc-\alpha_1ba-\alpha_2bc+1\\
  &\to_f\alpha_2b(c-1)-\alpha_1ba-\alpha_2bc+1\\
&=-\alpha_2b-\alpha_1ba+1,
\end{align*}
which is irreducible modulo $G$ and nonzero, contradicting that $bax-x+1\to_G0$.
Therefore, $bau_2\in\Irr(G)$.
Now the only possibility in which $\alpha_2bau_2-\alpha_1ba-\alpha_2u_2+1\to_G0$ is that $u_2=1$ and
$\alpha_1=\alpha_2=1$.
That is, $x=bca+1$.

Now we check by reduction whether $x=bca+1$ is a solution of (\ref{system}).
\begin{align*}
  bax-x+1&=ba(bca+1)-(bca+1)+1\\
  &=babca+ba-bca\\
  &\to_f b(c-1)a+ba-bca\\
  &=0\\
\end{align*}
Similarly, we have
\begin{align*}
  xba-x+1&=(bca+1)ba-(bca+1)+1\\
  &=bcaba+ba-bca\\
  &\to_g b(c-1)a+ba-bca\\
  &=0\\
\end{align*}
Thus, $x=bca+1$ is a solution of (\ref{system}) in $S$.
This implies that $x$ is an inverse of $1-ba$ in $S$.
Since an inverse is unique if it exists in $S$,
we do not need to consider Case (ii) for $x$.
\end{proof}
It is clear that Theorem \ref{thm_main} gives a new proof for Jacobson's lemma.

\

{\bf Acknowledgements.} This work was motivated by a talk by R. Padmanabhan in the Rings and Modules Seminar at the University of Manitoba.

\end{document}